\documentclass[10pt,a4paper,reqno]{amsart}
\usepackage{amsfonts,amsthm,latexsym,amsmath,amssymb,amscd,amsmath,epsf, 
graphicx}

\newtheorem{tm}{Theorem}

\newtheorem{rem}[tm]{Remark}

\newtheorem{lm}[tm]{Lemma}
\newtheorem{ex}[tm]{Example}

\newtheorem{prop}[tm]{Proposition}

\newtheorem{problem}[tm]{Problem}
\newtheorem{conj}[tm]{Conjecture}

\newcommand{\be}{\beta}

\newcommand{\si}{\sigma}

\newcommand{\bR}{\mathbb R}

\newcommand{\bZ}{\mathbb Z}
\newcommand{\D}{\mathcal D}

\newcommand{\eps}{\epsilon}

\newcommand{\bsi}{{\bar \si}}
\newcommand{\tsi}{{\widetilde \si}}

\begin{document}
\title{Could Ren\'e Descartes have known this?}%[Did Descartes knew this?]
%\author{Vladimir P. Kostov and Boris  Shapiro}

\author[J.~Forsg\aa rd]{Jens Forsg\aa rd}
\noindent 
\address{Department of Mathematics, Stockholm University, SE-106 91
Stockholm,
         Sweden}
\email{jens@math.su.se}

\author[V.~Kostov]{Vladimir P. Kostov }
\noindent 
\address{Universit\'e de Nice, Laboratoire de Math\'ematiques, 
Parc Valrose, 06108 Nice Cedex 2, France}
             \email{kostov@math.unice.fr}

\author[B.~Shapiro]{Boris Z. Shapiro}
\noindent 
\address{Department of Mathematics, Stockholm University, SE-106 91
Stockholm,
         Sweden}
\email{shapiro@math.su.se}



\begin{abstract}
Below  we discuss the partition of the space of real 
univariate polynomials according to the number of positive and negative  roots and signs of the coefficients.  We present several series of non-realizable combinations of signs together with the numbers of positive and negative roots.   We provide a detailed information   about possible non-realizable  combinations  as above up to degree $8$ as well as a general conjecture about such combinations. 
\end{abstract}

\date{}
\maketitle 

\dedicatory{Cogito ergo sum \quad (I think, therefore I am)}\footnote{The father of  modern philosophy and mathematician Ren\'e Descartes (in latin Renatus Cartesius) who spent most of his life in the Dutch Republic, died on 11 February 1650 in Stockholm, Sweden. He had been invited by Queen Christina of Sweden to tutor her. The cause of death was said to be pneumonia. One theory claims that accustomed to working in bed until noon, he may have suffered damage to his health from Christina's study regime, which began early in the morning at 5 a.m. %His lack of sleep could have severely compromised his immune system. 

Another theory says that  he might have been poisoned with arsenic for the following reason. At this time Queen Christina had intention  to convert to Catholicism and later she actually did that and abdicated her throne as Swedish law requires a Protestant ruler. The only Catholic with whom she had prolonged contact had been Descartes which might have caused the intense hatred by the Swedish Protestant clergy. On the other hand, another lead says that he might have been poisoned by a local Catholic priest who was afraid that Descartes radical religous ideas might interfere with Christina's intention to convert. In any case already in 1663, the Pope placed his works on the Index Librorum Prohibitorum (Index of Prohibited Books).

As a Catholic in a Protestant nation, he was interred in a graveyard used mainly for unbaptized infants in Adolf Fredriks kyrka in Stockholm. Later, his remains were taken to France and buried in the Abbey of Saint-Germain-des-Pr\'es in Paris. Although the National Convention in 1792 had planned to transfer his remains to the Panth\'eon, they are, two centuries later, still resting  in a chapel of the abbey. %His memorial, erected in the 18th century, remains in the Swedish church.

As if all that was not enough already, it seems that Descartes' body was shipped without his head due to the small size of the coffin  sent for him from France, and his cranium  was delivered to "Muse\'e de l'Homme`` in Paris only in early 19-th century by the famous Swedish chemist J.~Berzelius who bought it earlier  on an auction in Stockholm! One wonders if there is any reasonable evidence behind any of these  myths?
}

\section{Introduction}

The famous Descartes' rule of signs claims that the number of positive roots of a real univariate polynomial does not exceed the number of sign changes in its sequence of coefficients. In what follows we only consider polynomials with all non-vanishing coefficients. An arbitrary ordered sequence 
$\bsi=(\si_0,\si_1,...,\si_d)$ of $\pm$-signs  is called a {\em sign pattern}. Given a sign pattern ${\bsi}$ as above, we call by its {\em Descartes pair} $(p_\bsi,n_\bsi)$ the pair of non-negative integers counting sign changes and sign preservations of $\bsi$.  The Descartes pair of $\bsi$ gives the upper bound on the number of positive and negative roots respective of any polynomial of degree $d$ whose signs of coefficients are given by   $\bsi$. (Observe that, for any $\bsi$, 
$p_\bsi+ n_\bsi=d$.) To any polynomial $q(x)$ with the sign pattern $\bsi,$ we associate  the pair $(pos_q,neg_q)$ giving the numbers of its positive and negative roots (counting multiplicities). Obviously $(pos_q,neg_q)$  satisfies the standard restrictions
 \begin{equation}\label{stand} 
pos_q\le p_\bsi,\;  pos_q\equiv p_\bsi (mod\, 2),\; neg_q\le n_\bsi,\; neg_q\equiv n_\bsi (mod\, 2).
\end{equation}

We call pairs $(pos, neg)$ satisfying \eqref{stand} {\em admissible} for $\bsi$.  Conversely, for a given pair  $(pos, neg),$ we call a sign pattern $\bsi$ such that \eqref{stand} is satisfied {\em admitting} the latter pair. 
 It turns out that not for every pattern $\bsi,$ all its admissible  pairs  $(pos,neg)$  are realizable by polynomials with the sign pattern $\bsi$. Below we address this very basic question. 

\begin{problem}
For a given sign pattern $\bsi,$  which admissible pairs $(pos,neg)$ are realizable by polynomials whose signs of coefficients are given by  $\bsi$?  
\end{problem}

Consider the (affine) space $Pol_d$ of all real monic univariate polynomials of degree $d$ and define  the {\it standard real discriminant}  $\D_d\subset Pol_d$ as the subset of all polynomials having a real multiple root.  Detailed information about a natural stratification of $\D_d$ can be found in e.g., \cite {KhS}.  It is a  well-known and simple fact that $Pol_d\setminus \D_d$ consists of $\left[\frac{d}{2}\right]+1$ components distinguished  by the number of real  simple roots. Moreover, each such component is contractible. Strangely enough analogous  statements in the case when  one imposes  additional restrictions on the signs of coefficients seems to be unknown.  When working with monic polynomials we will mainly use their shortened signed patterns $\tsi$ representing the signs of all coefficients but the leading term; for the actual sign pattern $\bsi,$ we then write $\bsi=(1,\tsi)$ to emphasize that we consider  monic polynomials.

\medskip
To formulate our  results we need to introduce some notation.  For any pair $(d,k)$ 
of non-negative integers with $d-k\ge 0;\; d-k\equiv 0\text{\; mod\;} 2,$ denote by $Pol_{d,k}$,
the set of all monic real polynomials of degree $d$ with $k$ real  simple roots.  
Denote by $Pol^{\bsi}_d\subset Pol_d$ the set (orthant) of all polynomials $p=x^d+a_1x^{d-1}+...+a_d$ whose coefficients $(a_1,...,a_d)$ have the (shortened) sign pattern $\tsi=(\si_1,...,\si_d)$ respectively. Finally, set 
$Pol_{d,k}^\tsi=Pol_{d,k}\cap Pol_d^\tsi$. We have the natural action of $\bZ_2\times \bZ_2$ on the space of monic polynomials and on the set of all sign patterns respectively.  The first generator acts by reverting the signs of all monomials of odd degree 
(which for polynomials means $P(x)\to (-1)^dP(-x)$); 
the second generator acts reading the pattern backwards 
(which for polynomials means $P(x)\to x^dP(1/x)$). (To preserve the set of monic polynomials one has to divide $x^dP(1/x)$ by its leading term.)  We will refer to the latter action as the  {\em standard $\bZ_2\times \bZ_2$-action}.  % WHAT ABOUT CHAGING PLUSES TO MINUSES AND VICE VERSA?
(Up to some trivialities) the properties we will study below are invariant  under the standard $\bZ_2\times \bZ_2$-action. 

\medskip
We start with the following simple result (which should be known in the literature).  

\begin{tm} \label{th:real} {}
 
{\rm (i)}   If  $d$ is even, then $Pol_{d,0}^\tsi$ is nonempty  if and only if  $\si_d=+$ (i.e. the constant term is positive). 

\noindent
{\rm (ii)}    For any pair of positive integers $(d,k)$ with $d-k\ge 0$ and $d-k\equiv 0\text{\; mod\;} 2$  and any sign pattern $\tsi=(\si_1,...,\si_d),$  the set $Pol_{d,k}^\tsi$ is nonempty. 

\end{tm}

Observe that in general, the intersection  $Pol_{d,k}^\tsi$ does not have to be connected. The total  number $k$ of real zeros can be distributed between $m$ positive and $n$ negative in different ways satisfying the inequalities $m+n=k$, $m\le p_{\tsi},\; n\le n_{\tsi}$ and $m\equiv p_{\tsi} \mod 2$, $n\equiv n_{\tsi} \mod 2$. See examples below. 

On the other hand, some concrete intersections have to be connected. In particular, the following holds. 

\begin{prop}\label{lm:hypell}

\noindent 
{\rm (i)} For any $d$ and $\tsi$, the  intersections $Pol_{d,d}^\tsi$ and $Pol_{d,0}^\tsi$ are contractible. (The latter intersection is  empty for $d$ odd.)

\noindent 
{\rm (ii)} 
For the (shortened) sign pattern $\hat {+}=(+,+,\dots,+)$ consisting of all pluses,  the intersection $Pol_{d,k}^{\hat {+}}$ is contractible for any $k\le d,\; k\equiv d \mod 2$. 
  (The same holds for the shortened alternating sign pattern $(-,+,-,\dots)$.)

\noindent 
{\rm (iii)} For any sign pattern $\bsi=(1,\tsi)$ with just one sign change, all intersections $Pol^{\bsi}_{d,k}$ are  non-empty. For $k=d$ (which is the case of real-rooted polynomials having one positive and $d-1$ negative roots) this intersection  is  contractible.
\end{prop}

\medskip
Concerning non-realizable combinations of $\bsi$ and $(pos,neg),$ we have the following two statements. 

\begin{prop}\label{prop:even} For  $d$ even, consider patterns satisfying the conditions: the sign of
the constant term (i.e., last entry) is $+$; the signs of all odd monomials are $+$; among the remaining signs of even monomials  
there are $l\ge 1$ minuses (at arbitrary  positions).  Then, for any such sign pattern, the 
pairs $(2,0), (4,0), \dots , (2l,0),$ and only they, are non-realizable. (Using the standard 
$\bZ_2\times \bZ_2$-action one obtains more such examples.) 
\end{prop} 

\begin{problem}
 Does there exist a version of Proposition~\ref{prop:even} for odd degree $d$?
\end{problem}

%\begin{prop}\label{prop:odd} For  $d$ odd, consider the pattern $(+,-,-,\dots, -,-,+)$. The 
%pair $(0,d-2)$ is non-realizable for the latter pattern.
%\end{prop} 

\begin{prop}\label{prop:3parts} Consider a sign pattern $\bsi$ consisting of $m$ consecutive pluses (including the leading $1$) 
followed by $n$ consecutive minuses and then by $p$ consecutive pluses, where $m+n+p=d+1.$ Then 

\noindent 
{\rm (i)} for the pair $(0,d-2),$ this sign pattern is not realizable if 
\begin{equation}\label{eq:kappa}
\kappa=\frac{d-m-1}{m}\cdot \frac{d-p-1}{p} \ge 4;
\end{equation}

\noindent 
{\rm (ii)} the sign pattern $\bsi$ is realizable with any pair of the form $(2,v)$. 
\end{prop}

\begin{rem} {\rm Inequality~\eqref{eq:kappa} provides only sufficient conditions for non-realizability of the pattern 
$\bsi$ with the pair $(0,d-2)$. One can ask how sharp this condition is. But at the moment we do not have examples with 
\eqref{eq:kappa} violated when the pair $(0,d-2)$ is not realizable.  

%For example, 
%for $d=6$ the sign pattern $(1,-,-,-,-,+,+)$ is not realizable with the pair $(0,4)$, 
%see \cite {AlFu}, \S~5.  In this case $m=1,\;n=4$ and $\kappa=6$ On the other hand, 
%the sign pattern $(1,-,-,-,+,+,+)$ is realizable with the same pair, see Example~\ref{ex:6}. 
%In that case $m=1$, $n=3$ and $\kappa=8/3$.  %In both cases inequality~\eqref{eq:kappa} is violated. 
}
\end{rem}

 {\rm While  working on the project, we noticed a recent paper \cite{AlFu} dealing  with the same problem and giving  complete description of  non-realizable patterns and pairs $(pos,neg)$  for polynomials up to degree $6$. This paper contains interesting historical material as well as references  \cite {AJS}, \cite{Gr} to the earlier research in this topic.}   The main result of \cite{AlFu} is as follows.
 
 \begin{tm}\label{th:AlFu}
 {\rm (i)} Up to degree $d\le 4,$ for any sign pattern $\bsi$, all admissible pairs 
 $(pos,neg)$ are realizable; 
 
 \noindent
 {\rm (ii)} for $d=5$ (up to the standard $\bZ_2\times \bZ_2$-action) the only 
 non-realizable combination is $(1,-,-,-,+,+)$ with the pair $(0,3)$;
 
  \noindent
 {\rm (ii)} for $d=6$ (up to the standard $\bZ_2\times \bZ_2$-action) the only 
 non-realizable combinations are $(1,-,-,-,-,-,+)$ with $(0,2)$ and $(0,4)$; 
 $(1,+,+,+,-,+,+)$ with $(2,0)$; 
 $(1,+,-,-,-,-,+)$ with $(0,4)$.
  
 \end{tm}

Trying to extend Theorem~\ref{th:AlFu}, we obtained a computer-aided  classification of all non-realizable  sign patterns and pairs for $d=7$ and almost all for $d=8$, see below. 

\begin{tm}\label{th:7} For $d=7$,  among the 1472 possible combinations of a sign pattern
and a pair (up to the standard $\bZ_2\times \bZ_2$-action), there exist exactly 6 which are non-realizable. They are:
$$(1,+,-,-,-,-,-,+) \quad\text{with} \quad(0,5);\quad  (1,+,-,-,-,-,+,+) \quad\text{with} \quad(0,5);$$
$$(1,+,-,+,-,-,-,-) \quad\text{with} \quad(3,0);\quad  (1,+,+,-,-,-,-,+) \quad\text{with} \quad(0,5);$$
$$\text{and,}\quad(1,-,-,-,-,-,-,+) \quad\text{with}\;  (0,3)\;\text{and} \;(0,3).$$
\end{tm}

\begin{tm}\label{th:8}  For $d=8,$  among the 3648 possible combinations of a sign pattern and a pair (up to the standard $\bZ_2\times \bZ_2$-action), there exist 13 which are known to be non-realizable. They are:
$$(1,+,-,-,-,-,-,+,+) \quad\text{with} \quad(0,6);\quad  (1,-,-,-,-,-,-,+,+) \quad\text{with} \quad(0,6);$$
$$(1,+,+,+,-,-,-,-,+) \quad\text{with} \quad(0,6);\quad  (1,+,+,-,-,-,-,-,+) \quad\text{with} \quad(0,6);$$
$$(1,+,+,+,-,+,+,+,+) \quad\text{with} \quad(2,0);\quad  (1,+,+,+,+,+,-,+,+) \quad\text{with} \quad(2,0);$$
$$(1,+,+,+,-,+,-,+,+) \quad\text{with}\;  (2,0)\;\text{and} \;(4,0)\;;\quad(1,-,-,-,+,-,-,-,+) \quad\text{with}\;$$  
$$ (0,2)\;\text{and} \;(0,4); \quad(1,-,-,-,-,-,-,-,+) \quad\text{with}\;  (0,2), (0,4), \text{  and } (0,6).$$
\end{tm}
\begin{rem}\label{rem:8} {\rm For $d=8$, there exist 7 (up to the standard $\bZ_2\times \bZ_2$-action) 
combinations of a sign pattern and a pair for which it is still unknown whether they
are realizable or not. They are:
$$(1,+,-,+,-,-,-,+,+) \quad\text{with} \quad(4,0);\quad  (1,+,-,+,-,+,-,-,+) \quad\text{with} \quad(4,0);$$
$$(1,+,+,-,-,-,-,+,+) \quad\text{with} \quad(0,6);\quad  (1,+,+,-,-,+,-,+,+) \quad\text{with} \quad(4,0);$$
$$(1,+,+,+,-,+,-,-,+) \quad\text{with} \quad(4,0); 
\quad(1,+,-,+,-,-,-,-,+) \quad\text{with}\;  (4,0)$$ $$\;\text{and} \;(0,4).$$}
\end{rem}

Based on the above results, we formulate the following claim. 

\begin{conj}\label{conj:main} For an arbitrary sign pattern $\bsi$, the only type of pairs $(pos,neg)$ which can be non-realizable has either $pos$ or $neg$ vanishing. In other words, for any sign pattern $\bsi$, each pair $(pos,neg)$ satisfying \eqref{stand}  with positive $pos$ and $neg$ is realizable. 
\end{conj}

Rephrasing the above conjecture, we say that the only phenomenon implying non-realizability is that "real roots on one half-axis  force real roots on the other half-axis".  At the moment this conjecture is verified by computer-aided methods up to 
$d=10$.

\begin{tm}
\label{th:simplyconnectedk}
For any sign pattern $\bsi$ and integer $m$, the closure of the  union of intersections
\begin{equation}
\label{eqn:simplyconnectedset}
Pol_{d,\geq m}^\bsi := \bigcup_{j = m}^d \overline{Pol_{d,j}^\bsi}
\end{equation}
is simply connected.
\end{tm}

\medskip 
\noindent 
{\em Acknowledgements.} The second author is grateful to the Mathematics Department of Stockholm University for its hospitality in December 2009 and November 2014.  

\section{Proofs}

 The next two lemmas are very useful  in our arguments proving the realizability of a given pair $(pos, neg)$ with a given sign pattern $\bsi$.  %Suppose that a sign pattern $\bsi$ of degree $d$ is represented as the concatenation of two sign patterns of degrees $d_1$ and $d_2$ with $d_1+d_2=d$. For example. $(1,+,+,-,+,-,-,+,+)$ is a concatenation of $(1,+,+,-,+)$ and 

\begin{lm}\label{lm:concat}{\rm [}First concatenation lemma{\rm ]}  Suppose that  the monic polynomials $P_1$ and $P_2$ of degrees $d_1$ and $d_2$  with sign patterns  $\bsi_1=(1,\tsi_1 )$ and $\bsi_2=(1,
\tsi_2 )$ respectively  realize the pairs $(pos_1,neg_1)$ and $(pos_2,neg_2)$. (Here $
 \tsi_1$ and $ \tsi_2$ are the shortened sign patterns of $P_1$ and $P_2$ respectively.) Then 

\begin{itemize}

\item
if the last position of $\widetilde \si_1$ is $+,$ then for any $\eps>0$ small enough, the polynomial $\eps^{d_2}P_1(x)P_2(x/\eps)$ realizes the sign pattern $(1,\widetilde \si_1,\widetilde \si_2)$ and the pair $(pos_1+pos_2,neg_1+neg_2)$.

\item if the last position of $\widetilde \si_1$ is $-,$ then for any $\eps>0$ small enough, the polynomial $\eps^{d_2}P_1(x)P_2(x/\eps)$ realizes the sign pattern $(1,\widetilde \si_1,-\widetilde \si_2)$ and the pair $(pos_1+pos_2,neg_1+neg_2)$. (Here $-\tsi$ is the sign pattern obtained from $\tsi$ by changing each $+$ by $-$ and vice versa.)

\end{itemize}

\end{lm}

\begin{proof}  Set $P_2(x)=x^{d_2}+b_1x^{d_2-1}+b_2x^{d_2-2}+\dots+b_{d_2}$. Then $\eps^{d_2}P_2(x/\eps)=x^{d_2}+\eps b_1x^{d_2-1}+\eps^2 b_2x^{d_2-2}+\dots+\eps^{d_2}b_{d_2}$  and for $\eps>0$ small enough, the first $d_1+1$ coefficients of   $\eps^{d_2}P_1(x)P_2(x/\eps)$ are close to the respective coefficients of $P_1(x)$ (hence have the same signs). Then if the last entry of $\widetilde \si_1$ is $+$, the remaining coefficients of  $\eps^{d_2}P_1(x)P_2(x/\eps)$  (up to higher order terms in $\eps$) are equal to the respective coefficients of $\eps^{d_2}P_1(0)P_2(x/\eps)$. 
If this  entry  is $-$, the remaining coefficients of  $\eps^{d_2}P_1(x)P_2(x/\eps)$  (up to higher order terms in $\eps$) are equal to the opposite of the respective coefficients of $\eps^{d_2}P_1(0)P_2(x/\eps)$.
\end{proof}

\begin{ex}\label{ex:imp} {\rm  Denote by $\tau$ the last entry of $\widetilde \si_1$. We consider the cases $P_2(x)=x-1,\; x+1,\; x^2+2x+2,\; x^2-2x+2$ with $(pos_2,neg_2)=(1,0),\,(0,1),\,(0,0),\,(0,0)$. When $\tau=+,$ then one has respectively  $\widetilde \si_2=(-),\, (+),\, (+,+),\, (-,+)$ and the sign pattern of  $\eps^{d_2}P_1(x)P_2(x/\eps)$ equals $(1,\widetilde \si_1,-),$ $(1,\widetilde \si_1,+),$ $(1,\widetilde \si_1,+,+),$ $(1,\widetilde \si_1,-,+)$. When $\tau=-,$ then one has respectively  $\widetilde \si_2=(+),\, (-),\, (-,-),\, (+,-)$ and the sign pattern of  $\eps^{d_2}P_1(x)P_2(x/\eps)$ equals $(1,\widetilde \si_1,+),$ $(1,\widetilde \si_1,-),$ $(1,\widetilde \si_1,-,-),$ $(1,\widetilde \si_1,+,-)$.  }  
\end{ex}

\begin{ex}\label{ex:6}{\rm The sign pattern $(1,-,-,-,+,+,+)$ is realizable with the pair $(0,4)$. Indeed by Lemma~\ref{lm:concat} with $P_2(x)=x+1$, this follows from the realizability of the pattern $(1,-,-,-,+,+)$ for $d=5$ and the pair $(0,3)$ in which case one can set $P(x)=x(x^2-1)^2+\eps -\eps^2(x^2+x^4),$ where $\eps>0$ is small.  
}
\end{ex}

\begin{lm}\label{lm:concat2}{\rm [}Second concatenation lemma{\rm ]}
Take  (not necessarily monic) polynomials $P_1(x)=\sum_{k=0}^{d_1} a_k x^k$ 
and $P_2(x) = \sum_{k=0}^{d_2}b_{k}x^k$  of degrees $d_1$ and $d_2$ 
respectively  with all non-vanishing coefficients.  
Assume that they have sign patterns  $\bsi_1=(\widetilde {\si_1},+)$ and 
$\bsi_2=(+,\widetilde {\si_2})$ respectively and 
realize the pairs $(pos_1,neg_1)$ and $(pos_2,neg_2)$.  (Here $\widetilde \si_1$ 
and $\widetilde \si_2$ are arbitrary sequences of $\pm$ of lengths $d_1$ and $d_2$.)  Then,
for $\eps>0$ small enough, the polynomial
$$P(x) = \left(\frac{1}{a_{d_1}}\sum_{k=0}^{d_1-1} a_k x^k\right)
         + x^{d_1} + 
         \frac{x^{d_1}}{b_0}\left(\sum_{k=1}^{d_2}b_{k}(\epsilon x)^k\right)$$
realizes the sign pattern $(\widetilde \si_1,+,\widetilde \si_2)$ and 
the pair $(pos_1+pos_2,neg_1+neg_2)$.
\end{lm}

\begin{proof}
Since $a_{d_1},b_0 > 0$ by assumption, the polynomial $P$ has the sign pattern
$(\widetilde \si_1,+,\widetilde \si_2)$ for all $\epsilon > 0$. 
 Notice that, pointwise (and uniformly on compact subsets),
 $$P(x) \rightarrow \frac{P_1(x)}{a_{d_1}},\quad \epsilon \rightarrow 0, \quad  \text{and} \quad  \epsilon^{d_1}P(x/\epsilon)\rightarrow \frac{P_2(x)}{d_0}, \quad \epsilon \rightarrow 0.$$
 Therefore it is  clear that for $\epsilon$ sufficiently small, 
 $P$ has at least $pos_1+pos_2$ positive roots, and at least $neg_1+neg_2$ negative
 roots.
 
 It remains to show that for $\epsilon$ small enough, 
 the number of non-real roots of $P(x)$ is equal to the sum of the numbers of non-real  
 roots of $P_1$ and $P_2$. 
 By continuity of roots, for each neighborhood $N_p$ of a non-real root $p$ of $P_1$
 of multiplicity $m_p$,
 there is a $t = t(p) > 0$ such that $P(x)$ has $m_p$ roots in $N_p$ if $\epsilon < t(p)$.
 Similarly, for each neighborhood $N_q$ of a non-real root $q$ of $P_2$
 of multiplicity $m_q$ there is a $t = t(q)>0$ such that $P(x/\epsilon)$ has $m_q$ roots 
 in $N_q$. This implies that $P(x)$ has $m_q$ roots in the dilated set $\epsilon N_q$, 
 for $\epsilon < t(q)$.
 For each non-real root $p$ of $P_1$, 
 choose its neighborhood $N_p$ such that all $N_p$'s are pairwise disjoint and
  do not intersect the real axis. Choose the neighborhoods $N_q$ of
 the non-real roots $q$ of $P_2$ similarly. 
 If $P_1$ and $P_2$ has a common non-real root, then we cannot choose the neighborhoods $N_p$'s and $N_q$'s as above 
 so that $N_p$ is disjoint from $N_q$ for every pair $p$ and $q$. However,
 for $\epsilon$ sufficiently small, the dilated sets $\epsilon N_q$ are
 disjoint from $N_p$ for any $p$ and $q$. Indeed, since the open sets $N_p$
 do not meet $\bR$, there is a neighborhood $N_0$ 
 of the origin disjoint from each $N_p$;
  for $\epsilon$ small enough we have that $\epsilon N_q \subset N_0$ implying the latter claim.
 
 The fact that $N_q\cap \bR = \emptyset$, implies that 
 $\epsilon N_q \cap \bR = \emptyset$ as well. 
 Therefore, we can conclude that, for $\epsilon$ small enough, all roots of
 $P(x)$ contained in any of the sets $\epsilon N_q$ or $N_p$,
 are non-real, which finishes the proof.
\end{proof}

\begin{proof}[Proof of Theorem~\ref{th:real}]

Part (i) is straightforward. Indeed, the necessity of the positivity of the constant term is obvious for monic polynomials of even degree with no real roots.  Moreover fix any even degree monic polynomial with  coefficients of the necessary signs 
and increase its constant term  until the whole graph of the polynomial will lie  strictly above the $x$-axis. The resulting polynomial   has no real roots and the required signs of its coefficients.  

For the need of the rest of the proof, observe that in the same way one constructs  polynomials for $d$ odd which realize 
an arbitrary  sign pattern with exactly one real root (positive or negative depending on the sign of the constant term).    
As above one starts with an arbitrary odd degree polynomial with a given pattern and then one either increases or decreases the constant term until the polynomial has a single simple real root.

Now we prove part (ii) by induction on $d$ and $k$. For $d=1,2,3$ the fact can be easily checked. For
$k=0$ with $d$ even and $k=1$  with $d$ odd the proof is given above. 

Suppose first that $k=d$, i.e., the polynomial has to be real-rooted. In this case
one applies Lemma~\ref{lm:concat} and Example~\ref{ex:imp} $d-1$ times with $P_2=x\pm 1$. 

If $k<d$, then consider the last three signs of the sign pattern. If they are
$(+,+,+)$, $(+,-,+)$, $(-,-,-)$ or $(-,+,-)$, then one can apply Lemma~\ref{lm:concat} and
Example~\ref{ex:imp} with $P_2=x^2\pm 2x+2$. This preserves $k$ and reduces $d$ by $2$. 

Suppose that they are $(+,-,-)$ or $(+,+,-)$. If $d$ is odd, then one applies
Lemma~\ref{lm:concat} and Example~\ref{ex:imp} with $P_2=x\pm 1$ and reduces the proof to the case with
$d-1$ and $k-1$ in the place of $d$ and $k$. As $d$ is odd and $k>1$, one
actually has $k>2$. If $d$ is even, then one applies Lemma~\ref{lm:concat} and Example~\ref{ex:imp}
twice, with $P_2=x+1$ and with $P_2=x-1$ or vice versa. One obtains the case of
$d-2$, $k-2$. If $k-2=0$, then the proof of the theorem follows from part (i).
If $k>2$, then the reduction can continue.

Suppose that the last three signs are $(-,+,+)$ or $(-,-,+)$. If $d$ is odd and
$k=1$, the proof follows from part (i). If $d$ is odd and $k>2$, then one can
apply Lemma~\ref{lm:concat} and Example~\ref{ex:imp} with $P_2=x\pm 1$ and reduce the proof to the case
$d-1$, $k>0$. 

If $d$ is even and $k=0$, then the proof follows from part (i). If $d$ is even
and $k>0$, then one applies Lemma~\ref{lm:concat} and Example~\ref{ex:imp} with $P_2=x\pm 1$ and one
reduces the proof to the case $d-1$.
\end{proof}

To prove Proposition~\ref{lm:hypell}, we need the following lemma having an independent interest. 

\begin{lm}\label{lm:Jens}
 For any $\bar \sigma,$ the intersection $Pol_{d,d}^{\bar\sigma}$
 is path-connected.
\end{lm}

\begin{proof}
Recall that a real polynomial $p(x)$ is called sign-independently real-rooted if 
every polynomial obtained from $p(x)$ by an arbitrary sign change of its coefficients
is real-rooted. It is shown in \cite{PRS} that the logarithmic image of the set of 
all sign-independently real-rooted polynomials is convex.
 Hence the set of all sign-independently real-rooted polynomials itself is 
 logarithmically convex and in particular, it is path-connected. 
 The following criterion of sign-independently real-rootedness is straightforward. 

A real polynomial $p$ is sign-independently real-rooted if and only if, for every monomial $a_kx^k$ of $p(x),$ there exists a point $x_k$ such that
 \begin{equation}
 \label{eqn:lopsidedness}
  |a_kx_k^k| > \sum_{j\neq k}|a_jx_k^j|.
 \end{equation}
%  
% This claim implies that
%  $p(x)\in Pol_{d,d}^{\bar\sigma},$ 
%  for any choice $\bar \sigma$ of signs  of the coefficients of $p(x)$.

 Using induction on the degree $d,$ we will now prove  that, 
 for any polynomial $p\in Pol_{d,d}^{\bar\sigma},$ 
 there is a path $t\mapsto p_t$ such that
 (i) $p_0 = p$;  (ii) $p_1$ is sign-independently real-rooted; (iii)  $p_t \in Pol_{d,d}^{\bar\sigma}$ for all $t = [0,1]$.
 Since the set of all sign-independently real-rooted
 polynomials is path-connected, this claim settles Lemma~\ref{lm:Jens}.
 The case $d=1$ is trivial, as any linear polynomial
 is sign-independently real-rooted.
 
 Let $p$ be a real-rooted polynomial of degree $d$. Then, $q = p'$ is a 
 real-rooted polynomial of degree $d-1$. Hence, by the induction hypothesis,
 there is a path $t\mapsto q_t$ as above.
 Furthermore, since $p$ is real-rooted, so is its polar derivative
 $p'_\alpha(x) :=p(x) + \frac{x}{\alpha}p'(x)$
 for all $\alpha\in \bR^+$.
 
 For each $t\in [0,1]$, let $\alpha_t>0$ be such that
 $Q_{t,\alpha}(x) := p(x) + \frac{x}\alpha q_t(x)$ is
 real-rooted for any $0 <\alpha < \alpha_t$. By continuity
 of roots, $Q_{\hat t, \alpha_t}$ is real-rooted for $\hat t$ in
 a small neighborhood of $t$. Since $[0,1]$ is compact,
 we can find a finite set $\alpha_{t_1}, \dots, \alpha_{t_N}$
 such that $Q_{t,\alpha}(x)$ is real-rooted for all $t\in [0,1]$
 if $\alpha < \min(\alpha_{t_1}, \dots, \alpha_{t_N})$.
 
 Since $xq_1(x)$ is sign-independently real-rooted, for all $k$
 and all monomials $b_kx^k$ of $xq_1(x)$, there exists
 a point $x_k$ such that \eqref{eqn:lopsidedness} holds.
 Since the signs of $p(x)$ are equal to the signs of $xq_1(x)$,
 there exists an $\alpha_k>0$ such that \eqref{eqn:lopsidedness}
 holds for $Q_{1,\alpha}(x)$ for $k$ at $x_k$. However, since 
 \eqref{eqn:lopsidedness} always holds for the constant term with $x_0$ sufficiently small, we conclude that 
 $Q_{1,\alpha}$ is sign-independently real-rooted when  $\alpha < \min_{k=1, \dots ,d-1}\alpha_k $.

Now  fix a positive number 
  $\alpha^* < \min(\alpha_{t_1}, \dots, \alpha_{t_N}, \alpha_1, \dots \alpha_{d-1})$ and 
 consider the path composed of the two paths
 \[
  \alpha \mapsto p'_\alpha, \quad \alpha\in [\infty, \alpha^*]
 \qquad
\text{and}
 \qquad
  t \mapsto Q_{t,\alpha^*}, \quad t\in [0,1].
 \]
 By construction, this path is contained in $Pol_{d,d}^{\bar\sigma}$.
 Its starting point is $p(x)$ and its endpoint 
 $Q_{1,\alpha^*}$ is sign-independently
 real-rooted. This concludes the induction step. 
\end{proof}

\begin{proof}[Proof of Proposition~\ref{lm:hypell}]
To settle {\rm (i)} part 2 we notice that the set $Pol_{d,0}$
of all positive monic polynomials is a convex cone.  (Here $d$ is even.)
Therefore its intersection with any orthant is convex and contractible
(if nonempty).

To settle {\rm (i)} part 1, take a real-rooted polynomial $Q$
realizing a given
pattern. Consider the
family $Q+\lambda x^a$, $a=0,1,\ldots ,n-1$. Polynomials in this family  are real-rooted and with 
 the given sign pattern until either there is a confluence of roots of the polynomial, or   its  $a$-th
derivative vanishes at the origin. In both cases  further increase or decrease of the parameter 
$\lambda$ never brings us back to the set of real-rooted polynomials.

Thus the set $Pol^{\bar{\sigma}}_{d,d}$ has what we call  {\it Property A}: every its connected component   intersected with each line parallel to any coordinate axis
in the space of coefficients is either empty, or a point, or, finally,  an interval  whose endpoints  are
continuous functions of  other coefficients. (Indeed, they are values of the
polynomial or of its derivatives at roots of the polynomial or its derivatives; therefore 
these roots are algebraic functions 
depending continuously on the coefficients.)

Maxima and minima of such functions are also continuous.
Therefore the projection of each connected component of $Pol^{\bar{\sigma}}_{d,d}$ on each coordinate hyperplane
in the space of the coefficients 
also enjoys  { Property A}. (It suffices to fix the values of all coefficients
but one and study the endpoints of the segments as functions of 
that coefficient).

Now replace  $Pol^{\bar{\sigma}}_{d,d}$ by a smaller set obtained as follows. Choose some coefficient and, for fixed values of all other coefficients, substitute every nonempty intersection of $Pol^{\bar{\sigma}}_{d,d}$ 
with lines parallell to the  axis corresponding to the chosen coefficient  by the
half-sum of the endpoints, i.e., substitute the intersection segment by its middle point. This operation produces the graph of a continuous function
depending on the other coefficients. The projection of this graph to the coordinate hyperplane 
of  other coefficients is a domain having  Property A, but belonging to a
space of dimension $n-1$. Continuing  this process one contracts each connected component of the set $Pol^{\bar{\sigma}}_{d,d}$ to a point. 
Using Lemma~\ref{lm:Jens} we conclude that $Pol^{\bar{\sigma}}_{d,d}$ is 
path-connected and therefore contractible.

\medskip 
To prove {\rm (ii)}, it is enough to  settle the case
$\tsi=\hat + = (+,+,+,...,+)$.  Let us show that any compact subset in
$Pol_{d,k}^{\hat +}$ can be contracted to a point inside $Pol_{d,k}^{\hat +}$.
Observe that for any polynomial $p(x)$ with positive coefficients the family
of polynomials $p(x+t)$ where $t$ is an arbitrary positive number consists
of polynomials with all positive coefficients and the same number of 
real roots all being negative. Given a compact set $K\subset Pol_{d,k}^{\hat +}$, consider its
shift $K_t$ obtained by applying the above shift to the left on the distance
$t$, for $t$ sufficiently large. Then all real roots of all polynomials in the compact
set $K_t$ will be very large negative numbers and all complex conjugate pairs
will have very large negative real part. Therefore one can choose any specific
polynomial $\tilde p$ in $K_t$ and contract the  whole $K_t$ to $\tilde p$
along the straight segments, i.e., $\tau \tilde p+(1-\tau)p$ for any
$p\in K_t$. Obviously such contraction takes place  inside
$Pol_{d,k}^{\hat +}$.

Let us prove {\rm (iii)}.  It is clear that there is just one component (which is contractible) of real-rooted polynomials with all roots of the same sign. Suppose that they are all negative. To pass from degree $d$ to degree $d+1$ polynomials, from the pair $(0,d)$ to the pair $(1,d)$, one adds a positive root. One considers the polynomial
$$(x^d+a_1x^{d-1}+\cdots +a_d)(x-b), a_j>0, b\geq 0.$$
Its coefficients are of the form $c_j=a_j-ba_{j-1}, a_0=1$ (i.e., they are linear functions of the parameter $b>0$). Hence each of the coefficients except the first and the last one vanishes for some   $ b>0$ and then remains negative. As one must have for any $b>0$ exactly one sign change and never two consecutive zeros, it is always the last positive coefficient $c_j$ that vanishes. The value of $b$ for which a given coefficient $c_j$ vanishes depends  continuously on $a_i$  which implies the contractibility and uniqueness of the component with the pair $(1,d)$ with the different sign patterns.
 
We have just settled the real-rooted case with one sign change. Now we treat the non-real-rooted case. Fix a sign pattern with one sign change and with the pair $(1,d)$. One can realize it by a polynomial having all distinct  critical values. Hence when one decreases the constant term (it is negative, so the pattern does not change) the positive root goes to the right and the negative roots remain within a fixed interval $[-u,-v], u>0, v>0$.
When the constant term decreases, the polynomial loses consecutively
$[d/2]$ pairs of real negative roots and the realizable pairs become $(1,d-2), (1,d-4),...,(1,d-2[d/2])$ respectively.
%There remains to prove uniqueness of the components (for the cases $(1,d-2), (1,d-4),...,(1,d-2[d/2])$ ????.
\end{proof}

\medskip
Let us now prove the realizability for  a certain general class of pairs $(pos,neg)$.  
For a given sign pattern $\bsi,$ consider all possible sign patterns $\widetilde \bsi$ 
obtained from $\bsi$ by removing an arbitrary subset of its entries except for the 
leading $1$ and the last entry (constant term). On the level of polynomials
this corresponds to requiering that the corresponding coefficient vanishes.
For any such $\widetilde \bsi$, let $\widetilde {(pos,neg)}$ be its Descartes' pair, i.e., the number of its sign changes and the number of sign changes of the flip of $\widetilde\bsi$ (i.e., $P(-x)$). 

\begin{lm}\label{lm:stryk} Given an arbitrary sign pattern $\bsi$, 
all pairs $\widetilde {(pos,neg)}$   as above are realizable. 
\end{lm}

\begin{proof}
Recall that a sign-independently real-rooted polynomial is a real 
univariate polynomial such that it has only real roots and the 
same holds for an arbitrary sign change of its coefficients, 
see \cite {PRS}. As we already mentioned  a polynomial 
$p(x)=\sum_{k=0}^d a_kx^k$ is  sign-independently real-rooted if and 
only for each $k=0,\dots, d,$  there exists $x_k\in \bR_+$ such that 
\[
|a_kx_k^k| \ge  \sum _{l\neq k} |a_lx^l_k|.
\]
Let $P(x)$ be a 
sign-independently real-rooted polynomial with the given sign pattern 
$\bsi$. For each $\widetilde \bsi$, let $\widetilde P(x)$  denote 
the polynomial obtained by deleting those
 monomials from $P(x)$ which correspond to components of $\bsi$
 deleted when constructing $\widetilde \bsi$. Clearly the above inequality 
holds even for $\widetilde P(x)$ since we are removing monomials 
from its right-hand side. Therefore the sign of $\tilde p(x_k)$ 
equals that of $a_kx^k_k$. Since $x_0 < x_1 < \dots < x_d$, this 
implies that $\widetilde P(x)$ has at least $\widetilde {pos}$ sign 
changes in $\bR_+$. Similarly, we find that $\widetilde P(x)$ has
at least $\widetilde {neg}$ sign changes in $\bR_-$. However, by Descartes' rule of signs, this is the maximal number of positive and negative roots respectively. Hence, this is the exact number of positive and
negative roots of $\widetilde P(x)$. Therefore pertubations of the coefficients do not change the number of real roots.
\end{proof}

\begin{prop}\label{pr:Jens}
 Given an arbitrary sign pattern $\bsi$,
 any its admissible pair $(pos,neg)$ satisfying the condition
 $$\min(pos, neq) > \left\lfloor \frac{d-4}{3}\right\rfloor$$
 is realizable.
\end{prop}

\begin{proof}
 Notice first that, if $d\leq 3$, then
 $\left\lfloor \frac{d-4}{3}\right\rfloor < 0.$
 Thus we need to prove  that any admissible pair is realizable in this case.
 Indeed,  using Lemma \ref{lm:stryk}, this is straightforward to check.
 
 For arbitrary $d$, let us decompose $\bsi$ in the following manner.
 Let 
 $$\tau_k = (\sigma_{3k+1}, \dots, \sigma_{3k+4}),
 \quad k = 0, \dots, \left\lfloor \frac{d-4}{3}\right\rfloor,$$
 (where we use slight abuse of notation -- the last pattern needs not
 be of length $4$). Then, for each $\tau_k$, the admissible
 pairs are among the pairs
 $$(1,0), (1,2), (3,0), (0,1), (2,1), \text{ and, }(0,3),$$
 and for each $\tau_k$ all admissible pairs are realizable because they
 correspond to the case $d\le 3$.
 
 For each $\tau_k,$ associate initially an admissible
 pair $u_k = (1,0)$ or $u_k=(0,1)$ depending
 on whether $\tau_k$ admits an odd number of positive roots 
 and an even number of negative roots, or vice versa.
 By assumption,
 $$\sum_{k} u_k \leq (pos,neg)$$
 (where the inequality should be understood componentwise).
 If this is not an equality, then the difference is of the
 form $(2a,2b)$, where $a+b\leq \lfloor \frac{d-4}{3}\rfloor$,
 since the original pair $(pos,neg)$ is admissible.
 Define
 $$v_k = u_k + (2,0), \quad k =0, \dots, a-1,$$
 $$v_k = u_k + (0,2), \quad k = a, \dots, a+b-1,$$
 $$v_k = u_k, \quad k = a+b, \dots, \left\lfloor\frac{d-4}{3}\right\rfloor.$$
 Then, $v_k$ is an admissible pair for $\tau_k$, and in addition
 $$\sum_{k} v_k = (pos,neg).$$
 Applying Lemma \ref{lm:concat2} repeatedly to the patterns $\tau_k,$ we
 prove Proposition~\ref{pr:Jens}.
\end{proof}

\medskip
For $d$ odd, consider the sign patterns $\bsi=(1,\tsi)$ of the form: a) 
the last entry is $+$; b) all other entries at even positions are $-$; c) 
there is at most one sign change in the group of signs at odd positions.
Example, $(1,-,+,-,-,-,-,-,-,+)$.

\begin{lm}
\label{lm:jens2} 
Under the above assumptions if the pair has no positive roots, then it has exactly one negative, i.e., of all pairs $(0,s)$ only $(0,1)$ is realizable.
\end{lm}

\begin{proof} Let us decompose a polynomial $P(x)$ with the sign pattern $\bsi$ as above into $P_{od}(x)$ and $P_{ev}(x)$ where $P_{od}(x)$ (resp. $P_{ev}(x)$) contains all odd (resp. even) monomials of $P(x)$. Then obviously, $P_{ev}, P_{od}$ and $P^\prime_{od}$ have one positive root each which we denote by $x_{ev}, x_{od}$ and $x^\prime_{od}$ respectively. We first claim that $x^{\prime}_{od}<x_{od}<x_{ev}$.  Indeed, assume that $x_{ev}\le x_{od}$. Then both $P_{od}$ and $P_{ev}$ are non-positive on the interval $[x_{ev}, x_{od}]$.  Therefore also $P(x)$ would be non-positive on the same interval which contradicts to the assumption that  $P(x)$ is positive,  for all positive $x$. Now we prove that $x^{\prime}_{od}<x_{od}$. Present $P_{od}=P_{od}^+ - P_{od}^-,$ where $P_{od}^+$ is the sum of all odd degree monomials with positive coefficients and $P_{od}^-$ is the negative of the sum of all odd degree monomials with negative coefficients. Observe that the degree of the smallest monomial in $P_{od}^+$ is larger than $\delta=\deg P_{od}^-$ by assumption. 

Now if $P(x)\ge 0,$ i.e., $x\ge x_{od}$ then 
$$(P_{od}^+)^\prime(x) >\delta P_{od}^+(x)\ge \delta P_{odd}^-(x)>(P_{od}^-)^\prime(x)$$
which implies that $P^\prime_{od}(x)>0$, and hence $x>x^\prime_{od}$. 

Finally we  show that $P(x)$ has at most one negative root.  Consider the interval $[0,x_{od}]$. Since $x_{ev}>x_{od}, $ then $P_{ev}>0$ in $[0,x_{od}]$. Additionally,  $P_{od}$ is non-positive in this interval, implying that $P(-x)=P_{ev}(x)-P_{od}(x)$ is positive in the interval $[0,x_{od}]$. In the interval $[x^\prime_{od},+\infty)$, the polynomial $P^\prime_{od}$ is positive which together with the fact that $P^\prime_{ev}$ is negative implies that $P^\prime(-x)=P^\prime_{ev}-P^{\prime}_{od}$ is negative.   Thus being positive in $[0,x_{od}]$ and monotone decreasing to $-\infty$  in $[x^\prime_{od},+\infty)$, $P(-x)$ necessarily has exactly one positive root.  
\end{proof}
 
\begin{proof}[Proof of Proposition~\ref{prop:even}]  Suppose that a polynomial $P$ realizes a given sign pattern $\bsi$ with the pair $(2k,0),$ where $0<k\le l$. Then $P(0)>0$ and there exists $a>0$ such that $P(a)<0$. Hence $P(-a)<0$ since the monomials of even degree attain  the same value at $a$ and $-a$ while odd degree monomials have smaller values at $-a$ than at $a$ by our assumption on the signs. Thus there exists at least one negative root which is a contradiction. 

We prove that any pair of the form $(2s,2t)$, $0<s\le l,$ $0<t\le d/2 -l$ is realizable with the given sign pattern satisfying the assumption of Proposition~\ref{prop:even}. We make use of Lemma~\ref{lm:concat} and Example~\ref{ex:imp}. Represent the considered sign pattern $\bsi$ in the form $\bsi=(1,\widetilde \si_1, +)$, where the last entry of $\widetilde \si_1$ is $+$. (For $d=2$,  $\widetilde \si_1=+$.) Then the realizability of $\bsi$ with the pair $(2s,2t)$ follows from that of  $\bsi^\prime=(1,\widetilde \si_1)$ with the pair $(2s,2t-1),$ see Example~\ref{ex:imp} with $P_2(x)=x+1$.

 The sign pattern $\bsi^\prime$ has an even number of entries (including the leading  $1$).   Denote by $\bsi^\prime(r)$ the sign pattern obtained by truncation of the last $2r$ entries of $\bsi^\prime$.  (In particular, $\bsi^\prime(0)=\bsi^\prime$.) Below we will provide an algorithm which shows how realizability of  $\bsi^\prime(r)$ implies that of $\bsi^\prime(r-1)$. We also indicate how the pairs change in this process.  It is clear that $\bsi^\prime(d/2-1)=(1,+)$.  The corresponding pair is $(0,1)$ and it is realizable by the polynomial $x+1$.
 
 The sign pattern $\bsi^\prime$ contains $l$ minuses all of which occupy only odd positions. We mark the leftmost $s$ of them.  We distinguish between the following three cases according to the last two entries of $\bsi^\prime(r-1)$. Notice that the last entry of $\bsi^\prime(r)$ is always $+$.
 
 Case a) $\bsi^\prime(r-1)=(\bsi^\prime(r),-,+)$ and the  minus sign in the last but one position is marked. In this case one applies Example~\ref{ex:imp} twice,  each time  with $P_2(x)=x-1$. If the pair of $\bsi^\prime(r)$ equals $(u,v),$ then the one of  $\bsi^\prime(r-1)$ equals $(u+2,v)$.  
 
 Case b) $\bsi^\prime(r-1)=(\bsi^\prime(r),-,+)$ and the  minus sign in the last but one position is not marked. In this case one applies Example~\ref{ex:imp}   with $P_2(x)=x^2-2x+2$. The pairs of $\bsi^\prime(r-1)$ and of $\bsi^\prime(r)$ are the same.   
  
 Case c) $\bsi^\prime(r-1)=(\bsi^\prime(r),+,+)$. If $v<2t,$ then one applies Example~\ref{ex:imp} twice, each time with $P_2(x)=x+1$. The pair of $\bsi^\prime(r-1)$ equals $(u,v+2)$. If $v= 2t$, then one applies Example~\ref{ex:imp} with $P_2(x)=x^2+2x+2$. The pairs of $\bsi^\prime(r-1)$ and of $\bsi^\prime(r)$ are the same.  
  
 Observe that any sign pattern $\bsi$ satisfying the assumptions of Proposition~\ref{prop:even} can be obtained from the initial $(1,+)$ by applying consecutively $d/2-1$ times the appropriate of the Cases a) -- c).  Notice that we add exactly $2s$ positive and $2(t-1)$ negative roots. Another negative root comes  from $\bsi^\prime(d/2-1)$ and the last one is obtained  when passing from $\bsi^\prime$ to $\bsi$. Hence the pattern and the pair are realizable. 
 \end{proof}
 
 \begin{proof}[Proof of Proposition~\ref{prop:3parts}]  To prove  (i) we show that the three-part sign 
 pattern $\bsi$ satisfying the assumptions of Proposition~\ref{prop:3parts}, is not realizable by a 
 polynomial $P(x)$ having $d-2$ negative and a double positive root. By a linear change of $x$ the latter 
 can be assumed to be equal to $1$: 
 $$P(x)=(x^2-2x+1)S(x),\quad \text{where}\quad S(x)=x^{d-2}+a_1x^{d-3}+\cdots+a_{d-2}.$$
 Here $a_j>0$ and the factor $S(x)$ has $d-2$ negative roots. The coefficients of $P(x)$ are equal to
 $$1, a_1-2, a_2-2a_1+1, a_3-2a_2+a_1, \dots, a_{d-2}-2a_{d-3}+a_{d-4},-2a_{d-2}+a_{d-3},a_{d-2}.$$
 We want to show that it is impossible to have both inequalities: 
 $$a_m-2a_{m-1}+a_{m-2}<0\quad (*)\quad\text{and}\quad a_{m+n-1}-2a_{m+n-2}+a_{m+n-3}<0\quad(**)$$
 satisfied. 
 
 Now consider a polynomial having $d-2$ negative roots and a complex conjugate pair. If the polynomial has 
 at least one negative coefficient, then its factor having complex roots must be of the form $x^2-2\be x+\be^2+\gamma$, 
 where $\be>0$ and $\gamma >0$. A linear change of $x$ brings the polynomial to the form 
 $$Q(x)=(x^2-2x+1+\delta)S(x), \;\delta >0.$$ 
 The coefficients of $Q(x)$ are obtained from that of $P(x)$ by adding the ones of the polynomial 
 $\delta S(x)$. If inequality $(**)$ fails, then the coefficient of $x^{d-m-n+1}$ in $Q(x)$ is positive 
 (it equals $a_{m+n-1}-2a_{m+n-2}+a_{m+n-3}+\delta a_{m+n+1}>0$). So the sign pattern of $Q(x)$ is different 
 from $\bsi$.  If inequality $(**)$ holds, then inequality $(*)$ fails and the coefficient of $x^{d-m}$ in 
 $Q(x)$ is non-negative, so $Q(x)$ does not have  the sign pattern $\bsi$. 
 
 The polynomial $S(x)$ being real-rooted, its coefficients satisfy the Newton inequalities: 
%  $$a_k^2\ge \frac{k+1}{k}\cdot \frac{d-k+1}{d-k}a_{k-1}a_{k+1}, \; k=1,\dots , d-3, \quad 
%  (\text {we set}\quad a_0=1).$$
 $$\frac{a_k^2}{{d-2 \choose k}^2} \ge \frac{a_{k-1}a_{k+1}}{{d-2 \choose k+1}{d-2\choose k-1}},\quad \; k=1, \dots, d-3
 \quad (\text {we set}\quad a_0=1).$$
 
%  Here $$\kappa=\prod_{k=m+1}^{m+n-2}  \frac{k+1}{k}\cdot \frac{d-k+1}{d-k}=\frac{m+n-1}{m+1}\cdot \frac{d-m}{d-m-n+2},$$
%  i.e., $a_ma_{m+n-3}\ge \kappa a_{m-1}a_{m+n-2}.$ 
   Here $$\kappa = \frac{{d-2\choose m}{d-2\choose m+n-3}}{{d-2\choose m-1}{d-2\choose m+n-2}} = \frac{d-m-1}{m}\cdot\frac{d-p-1}{p},$$
   i.e., $a_ma_{m+n-3}\ge \kappa a_{m-1}a_{m+n-2}.$ 
   Inequalities $(*)$ and $(**)$ imply respectively $$a_m<2a_{m-1}\quad  \text{and}\quad a_{m+n-3}<2a_{m+n-2}.$$ 
   Thus $a_ma_{m+n-3}\ge \kappa a_{m-1}a_{m+n-2} > \frac{\kappa}{4} a_{m}a_{m+n-3}$, 
   which is a contradiction since $\kappa \geq 4$ by assumption.   
   
   To prove (ii) we use Lemma~\ref{lm:concat} and Example~\ref{ex:imp}.   We construct sign patterns $\bsi(0)=\bsi, \bsi(1), \dots ,$ each being a truncation of the previous one (by one, two or three according to the case as explained below),  and corresponding pairs $(u_j, v_j)$, $j=0,1,2,\dots ,$ where $(u_0,v_0)=(2,v),$ such that realizability of $\bsi(j-1)$ with $(u_{j-1},v_{j-1})$ follows from the one of $\bsi_j$ with $(u_j,v_j)$. For convenience we write instead of the pair $(u_j,v_j)$ the triple $(u_j,v_j,w_j)$, where $w_j$ is the number of complex conjugate pairs of roots (hence $u_j+v_j+2w_j=d_j$, where $d_j+1$ is the number of entries of $\bsi(j))$. 
   
   We consider first the case $v\ge 2$. The necessary modifications in the cases $v=0$ and $v=1$ are explained at the end of the proof. 
   
   If $\bsi(j-1)$ has not more than 3 entries, then we do not need to construct the sign pattern $\bsi(j)$. Two cases are to be distinguished: 
   
   \medskip
   Case A. If $\bsi(j-1)$ has only 2 entries, then these are either $(1,+)$ or $(1,-)$, and $\bsi(j-1)$ is realizable respectively by the polynomials $x+1$ or $x-1$ with $(u_{j-1},v_{j-1},w_{j-1})=(0,1,0)$ or $(1,0,0)$.

   \medskip 
   Case B. If $\bsi(j-1)$ has only 3 entries, then they can be only $(1,+,+)$, $(1,+,-)$ or $(1,-,-)$. In the first case $(\bsi(j-1),(0,2,0))$ is realizable by $(x+1)(x+2)$ and $(\bsi(j-1), (0,0,1))$ by $x^2+2x+2$.  In the second case  $(\bsi(j-1),(1,1,0))$ is realizable by $(x+2)(x-1)$ and in the third case $(\bsi(j-1),(1,1,0))$ by $(x+1)(x-2)$.

Suppose that $\bsi(j-1)$ contains more than 3 entries. The following cases are possible: 

\medskip
Case C. 
The last three entries of $\bsi(j-1)$ are $(+,+,+)$ or $(-,-,-)$. If $w_{j-1}>0$, then we apply Lemma~\ref{lm:concat} with $P_2(x)=x^2+2x+2$ and we set $(u_j,v_j,w_j)=(u_{j-1},v_{j-1},w_{j-1}-1)$. If $w_{j-1}=0$, then we apply Lemma~\ref{lm:concat} twice, both times  with $P_2(x)=x+1$.  We set $(u_j,v_j,w_j)=(u_{j-1},v_{j-1}-2,0).$ 

\medskip
Case D. The last three entries of $\bsi(j-1)$ are $(-,+,+)$ or $(+,-,-)$. One applies Lemma~\ref{lm:concat} twice, the first time with $P_2(x)=x+1$ and the second time with $P_2(x)=x-1$. One sets $(u_j,v_j,w_j)=(u_{j-1}-1,v_{j-1}-1,w_j)$.

\medskip
Case E. The last three entries of $\bsi(j-1)$ are $(-,-,+)$ or $(+,+,-)$.  One applies Lemma~\ref{lm:concat}  with $P_2(x)=x-1$
 and sets $(u_j,v_j,w_j)=(u_{j-1}-1,v_j,w_j)$. 
 
 In Cases C and D, $\bsi(j)$ has two entries and in Case E it has one entry less than $\bsi(j-1)$. 
 
 Further explanations. In the pair $(2,v)$ obtained as the result of this algorithm the first component equals 2 because one encounters exactly once Case D with $(-,+,+)$ or Case E with $(-,-,+)$ (in both of them $u_{j-1}$ decreases by 1); and exactly once  Case D with $(+,-,-)$ or Case E with $(+,-,-)$ (when $u_{j-1}$ also decreases by 1); or Case A with $(1,-)$ or Case B with $(1,+,-)$ or $(1,-,-)$.
 
 When $v=0$ or $v=1,$ one does not have the possibility to apply Lemma~\ref{lm:concat} with $P_2(x)=x+1$ and Cases D and E have to be modified. One has to consider the last 4 entries of $\bsi(j-1)$. If they are $(+,\pm,\mp,-)$ or $(-,\mp,\pm,+)$, then one applies Lemma~\ref{lm:concat} with $P_2(x)=x^3\pm \eps_1 x^2\mp \eps_2 x-1$, where $\eps_i>0$ are small. One sets $(u_j,v_j,w_j)=(u_{j-1}-1,v_{j-1},w_{j-1}-1)$ and $\bsi(j)$ has three entries less than $\bsi(j-1)$. 
 \end{proof}

\begin{proof}[Proof of Theorem~\ref{th:7}]
 The fact that the patterns given in the formulation of  Theorem~\ref{th:7}  are non-realizable
 follows from Proposition \ref{prop:3parts} and Lemma \ref{lm:jens2}.
  It remains to show that all other admissible patterns and
 pairs are realizable.
 
 Using Lemma \ref{lm:concat2} and a Mathematica script \cite{JensSkript} written by the first author, this question is reduced to checking
 the cases: 
 $$(1,+,-,+,+,-,+,+) \quad\text{with} \quad(4,1);\quad  (1,+,+,-,-,-,+,+) \quad\text{with} \quad(0,5);$$
 $$(1,+,+,-,-,+,-,-) \quad\text{with} \quad(3,0);\quad  (1,+,+,+,-,-,-,+) \quad\text{with} \quad(0,5);$$
 $$(1,+,+,+,-,+,+,-) \quad\text{with} \quad(3,0);\quad  (1,+,-,+,-,+,+,-) \quad\text{with} \quad(3,0);$$
 $$(1,+,-,+,+,+,-,-) \quad\text{with} \quad(3,0);\quad\text{and}\quad  (1,+,-,+,+,+,+,-) \quad\text{with} \quad(3,0).$$
 The first five cases can be settled by using either Lemma \ref{lm:concat} or Lemma \ref{lm:concat2}.
 For the realizability of the remaining three cases, we provide the following concrete examples:
 $$P_1(x) = (x - 0.1690) (x - 1.4361) (x - 2.0095) (x^2 + 0.0218 x + 
   6.2846) (x^2 + 3.6029 x + 3.2609),$$
 $$P_2(x) = (x - 2.6713) (x - 2.6087) (x - 0.6059) (x^2 + 0.5495 x + 
   0.3304) (x^2 + 5.3464 x + 7.1668),$$
 $$P_3(x) = (x - 0.6056) (x - 2.6105) (x - 2.6696) (x^2 + 0.5493 x + 
   0.3305) (x^2 + 5.3465 x + 7.1672).$$
\end{proof}

\begin{proof}[Proof of Theorem~\ref{th:8}]
The fact that the patterns given in the formulation of  Theorem~\ref{th:7}  are non-realizable
 follows from Proposition \ref{prop:3parts} and Lemma \ref{lm:jens2} except for the case
 \[
 (1,-,-,-,+,-,-,-,+) \quad \text{with } (0,2) \text{ and } (0,4).
 \]
 Substituting $x \mapsto -x$, we obtain the sign pattern
 $(1,+,-,+,+,+,-,+,+)$. That is $P(x) < P(-x)$ for $x\in \bR_+$.
 In particular, if $P$ has a negative root, 
 then it has at least two positive roots.
 
 Using Lemma~\ref{lm:concat2} and Lemma~\ref{lm:stryk} and the above mentioned Mathematica script, all patterns
 except those of Remark \ref{rem:8} can be shown to be realizable.
\end{proof} 

\begin{proof}[Proof of Theorem~\ref{th:simplyconnectedk}]
We will follow the steps of the proof of Lemma \ref{lm:Jens}.  
For any polynomial $p$, the set $K_p$ consisting of
all exponents $k$ such that there exists a $x_k\in \bR_+$ for which 
\eqref{eqn:lopsidedness} holds,
provides a lower bound on the number of real roots of $p$. 
This lower bound is called the number of 
\emph{lopsided induced zeros} of $p$. 
Fixing an arbitrary set of exponents  $K$, 
let us denote by $S_K$ the set of all polynomials such that  $K\subseteq K_p$. 
It is shown in \cite{F14} that $S_K$ is logarithmically convex.
For example, if $K = \{0,1,\dots,d\}$, then $S_K$ is the set of all sign-independently
real-rooted polynomials.
Consider the family $F_m$ consisting of all exponent sets $K$ such that
the number of lopsided induced zeros of polynomials in $S_K$
is at least $m$. 
The set $S_m = \cup_{K\in F_m} S_K$ is a union of
logarithmically convex sets, whose intersection contains the
set of all sign-independently real-rooted polynomials.
In particular, $S_m$ is path-connected.

As in the proof of Lemma \ref{lm:Jens}, 
 for any polynomial $p$ which has at 
least $m$ real roots, 
all polynomials in the path
\[
\alpha \mapsto p(x) + \frac{x}{\alpha} p'(x), \quad \alpha\in [\infty, \alpha^*]
\]
have at least $m$ real roots. Exactly the same argument as in 
the proof of Lemma \ref{lm:Jens} gives path-connectedness
of the set \eqref{eqn:simplyconnectedset} of polynomials with at least $m$ real roots.

Let us now prove simply connectedness of the set \eqref{eqn:simplyconnectedset} 
by induction on the degree $d$. Consider a closed loop in $Pol_{d,\geq m}^\bsi$, 
i.e., a path $\ell$ given by $\theta \mapsto p_\theta(x)$, $\theta\in[0,1]$, such 
that $p_0(x) = p_1(x)$, and such that $p_\theta(x)$ has at 
least $m$ real roots for all $\theta$.

Consider the induced loop $\ell'$ given by 
$\theta \mapsto p'_\theta(x)$, where we use the notation
$p'_\theta(x) = \frac{d}{dx}p_\theta(x)$.
It is contained in the set $Pol_{d-1,\geq m-1}^{\hat\bsi}$,
where $\hat\bsi$ is obtained from $\bsi$ by deleting
its last entry. 
By the induction hypothesis, the loop $\ell'$ can be contracted 
to a point within the set of all polynomials of degree $d-1$ 
with at least $m-1$ real roots. 
In other words, we have a map $(\theta, \phi) \mapsto p'_{(\theta,\phi)}$,
for $(\theta,\phi)\in [0,1]^2$, satisfying the conditions: 
($i$) $p'_\theta(x) = p'_{(\theta,0)}(x)$; ($ii$) $p'_{(\theta, 1)}$ 
is independent of $\theta$; ($iii$)
$p'_{(\theta,\phi)}$ has at least $m-1$ real 
roots for all $\theta$ and $\phi$. 
The last property implies that $xp'_{(\theta,\phi)}$
has at least $m$ real roots for all $\theta$ and $\phi$.
Define $p_{(\theta, \phi)}$ by the conditions that
$\frac{d}{dx}p_{(\theta, \phi)} = p'_{(\theta, \phi)}$ and
that the constant term of $p_{(\theta, \phi)}$ is independent of $\phi$.

Since the loop $\ell'$ is compact, we can find an 
$\alpha^*\in \bR_+$ such that the polar derivative
\[
p'_{(\theta,\phi,\alpha)}(x) := p_{(\theta, \phi)}(x) + \frac{x}{\alpha} p'_{(\theta, \phi)}(x)
\]
has at least $m$ roots for each $\alpha < \alpha^*$ and all 
$(\theta, \phi) \in [0,1]^2$. 
Thus, similarly to the proof of Lemma \ref{lm:Jens}, the composition
of the maps 
\[
 \alpha \mapsto p'_{(\theta,0,\alpha)}, \quad \alpha\in [\infty, \alpha^*]
 \qquad
 \text{and}
 \qquad
 \phi \mapsto p'_{(\theta,\phi,\alpha^*)}, \quad \phi\in [0, 1]
\]
provides a contraction of the loop $\ell$
in the set $Pol_{d,\geq m}^\bsi$.
\end{proof}

%\begin{rem}
%{\rm In the previous proof, we used only the fact that loop $\ell$ is compact.
%It is not necessary to assume that is is one dimensional.
%Indeed, a similar argument gives that, for any integer $l$, all $l$-cycles $\ell$ are 
%contractible, 
%\emph{if} one proves the induction basis, namely that all 
%$l$-cycles $\ell$ in the space of polynomials with degree $l+2$ 
%and at least $m$ real roots, are contractible.}
%\end{rem}

\section{Final Remarks} 

Above we mainly discussed the question  which pairs $(pos,neg)$ of the numbers of positive and negative roots satisfying the obvious compatibility conditions are realized by polynomials with a given sign pattern. Our main Conjecture~\ref{conj:main} presents restrictions observed in consideration of all non-realizable pairs up to degree $10$. However the following  important and closely related questions remained unaddressed above.  

\begin{problem}
Is the set of all polynomials realizing a given pair $(pos,neg)$ and having a sign pattern $\bsi$ path-connected (if non-empty)?  
\end{problem}

Given a real polynomial $p$ of degree $d$ with all non-vanishing coefficients, consider  the sequence of pairs 
$$\{(pos_0(p),neg_0(p)), (pos_1(p),neg_1(p)), (pos_2(p),neg_2(p)), \dots, (pos_{d-1}(p),neg_{d-1}(p))\},$$ where
$(pos_j(p),neg_j(p))$ is the numbers of positive and negative roots of $p^{(j)}$ respectively.  Observe that if one knows the above sequence of pairs then one knows the sign pattern of a polynomial $p$ which is assumed to be monic.  Additionally it is easy to construct examples that the converse fails. 

\begin{problem}
Which sequences of pairs are realizable by real polynomials of degree $d$ with all non-vanishing coefficients? 
\end{problem}
Notice that similar problem for the sequence of pairs of real roots (without division into positive and negative) was considered in \cite{Ko}. 

\medskip
Our final question is as follows. 

\begin{problem}
Is the set of all polynomials realizing a given sequence of pairs  as above path-connected (if non-empty)?  
\end{problem}

\end{document}